\documentclass[12pt,reqno]{amsart}
\usepackage[pagebackref,colorlinks=true]{hyperref}
\usepackage{float}
\usepackage{changepage}
\usepackage{verbatim}
\usepackage{amsfonts,amsmath,amsthm,amssymb,amscd,enumerate,eucal,url}
\usepackage{mathtools}
\usepackage[margin=1in]{geometry}
\usepackage{xcolor}
\usepackage{tikz}
\usepackage{microtype}
\usetikzlibrary{arrows.meta,positioning}

\usepackage{mathrsfs}

\numberwithin{equation}{section}
\newlength{\extralength}
\newcommand{\R}{\mathbb{R}}

\newtheorem{Theorem}{Theorem}[section]
\newtheorem{Lemma}{Lemma}[section]
\newtheorem{Proposition}{Proposition}[section]

\newtheorem{Remark}{Remark}[section]

\newtheoremstyle{axiomstyle}
  {}{}                 % razmaci
  {\itshape}           % telo
  {}                   % indent
  {\bfseries}          % naslov
  {.}                  % tačka iza broja
  {0.5em}              % razmak
  {\thmname{#1}~\thmnumber{#2}\thmnote{ (#3)}} % OVO PRIKAZUJE (RG0...)
  
\theoremstyle{axiomstyle}

\title{Multidimensional cost geometry}
\date{\today}

\author{Jonathan Washburn}
\address{Recognition Physics Institute, Austin, Texas, USA}
\email{jon@recognitionphysics.org}

\author{Milan Zlatanovi\'c}
\address{Department of Mathematics, Faculty of Science and Mathematics, University of Ni\v s, Vi\v segradska 33, 18000 Ni\v s, Serbia}
\email{zlatmilan@yahoo.com}

\author{Philip Beltracchi}
\address{Recognition Physics Institute, Austin, Texas, USA}
\email{philipbeltracchi@gmail.com}

\begin{document}

\newcommand{\add}[1]{{\color{blue}#1}}
\newcommand{\del}[1]{{\color{red}#1}}
\newenvironment{added}{\begingroup\color{blue}}{\endgroup}
\newenvironment{deleted}{%
  \begingroup\color{red}%
  % Prevent deleted text from affecting citation order/numbering.
  \renewcommand{\cite}[2][]{\relax}%
  % Prevent deleted text from affecting labels/counters/structure.
  \renewcommand{\label}[1]{\relax}%
  \let\ref\relax
  \let\eqref\relax
  \let\pageref\relax
  \renewcommand{\section}[1]{\par\medskip\noindent{\bfseries [deleted section] ##1}\par}%
  \renewcommand{\subsection}[1]{\par\medskip\noindent{\bfseries [deleted subsection] ##1}\par}%
  \renewcommand{\subsubsection}[1]{\par\medskip\noindent{\bfseries [deleted subsubsection] ##1}\par}%
  \renewcommand{\paragraph}[1]{\par\noindent{\bfseries [deleted paragraph] ##1}\par}%
  \renewcommand{\subparagraph}[1]{\par\noindent{\bfseries [deleted subparagraph] ##1}\par}%
  \renewenvironment{equation}{\[\ignorespaces}{\]\ignorespacesafterend}%
  \renewenvironment{theorem}[1][]{\par\medskip\noindent\textbf{[deleted theorem] }\ignorespaces}{\par\medskip}%
  \renewenvironment{lemma}[1][]{\par\medskip\noindent\textbf{[deleted lemma] }\ignorespaces}{\par\medskip}%
  \renewenvironment{proposition}[1][]{\par\medskip\noindent\textbf{[deleted proposition] }\ignorespaces}{\par\medskip}%
  \renewenvironment{corollary}[1][]{\par\medskip\noindent\textbf{[deleted corollary] }\ignorespaces}{\par\medskip}%
  \renewenvironment{definition}[1][]{\par\medskip\noindent\textbf{[deleted definition] }\ignorespaces}{\par\medskip}%
  \renewenvironment{remark}[1][]{\par\medskip\noindent\textbf{[deleted remark] }\ignorespaces}{\par\medskip}%
  \renewenvironment{example}[1][]{\par\medskip\noindent\textbf{[deleted example] }\ignorespaces}{\par\medskip}%
  \renewenvironment{notation}[1][]{\par\medskip\noindent\textbf{[deleted notation] }\ignorespaces}{\par\medskip}%
  \renewenvironment{convention}[1][]{\par\medskip\noindent\textbf{[deleted convention] }\ignorespaces}{\par\medskip}%
  \renewenvironment{axiom}[1][]{\par\medskip\noindent\textbf{[deleted axiom] }\ignorespaces}{\par\medskip}%
}{\endgroup}

\newcommand{\Poly}{\R[u,v]}

 \begin{abstract}In this paper, we study the geometric structure induced by the canonical reciprocal cost function and its natural $n$-dimensional extension. In logarithmic coordinates, the potential depends only on the linear combination $S=\alpha\cdot t$, and the associated Hessian metric has rank one at every point. The geometry is intrinsically degenerate and effectively one-dimensional, with an $(n-1)$-dimensional null distribution.
On the other hand, when the same function is expressed in the original $x$-coordinates, the corresponding Hessian is generically nondegenerate and defines a pseudo-Riemannian metric %(see e.g. \highlighting{\cite{ONeill1983}}) 
away from explicit singular hypersurfaces. 
We further analyze affine and Levi-Civita geodesics and compare their behavior. In particular, affine geodesics in logarithmic coordinates are globally defined, while in $x$-coordinates their behavior is restricted by the domain and the singular set.
Finally, we relate the construction to symmetrized Itakura--Saito and Bregman divergences, and give a Fisher--Rao realization of the logarithmic Hessian metric.

\bigskip


\medskip 


\end{abstract}
\maketitle

\setcounter{tocdepth}{3}

%\tableofcontents

\newcommand{\config}{\mathcal{C}}
\newcommand{\configR}{\mathcal{C}_R}

\section{Introduction and Motivation}

In this paper, we study the geometric structures generated by the canonical reciprocal cost function and its multidimensional extension. The canonical cost function in \mbox{one-dimension}
\[
J(x)=\frac12(x+x^{-1})-1
\] 
was introduced in~\cite{WZ}. Generic cost functions are ubiquitous in optimization techniques, and different cost functions can have different motivations. In the paper~\cite{WZE}, it is proved that this particular function
appears as a unique solution of  the polynomial composition law together with the curvature calibration.

In this paper, we construct the multidimensional extension of the reciprocal cost, examine what kind of multidimensional
geometry is induced by the Hessian of the function $J$, and explore how this geometry
depends on the choice of affine structure. Hessian geometry appears on affine manifolds with a flat connection. 
A metric $g$ is called Hessian if it can be written locally as $g = Dd\varphi$ for some potential $\varphi$ ~\cite{Amari, shima2007}. Here $D$ denotes a flat affine connection, so in affine coordinates the metric takes the form
$
g_{ij} = \partial_i \partial_j \varphi.
$
The pair $(D,g)$ is called a Hessian structure. Such structures are related to several geometric theories. 
They are real analogs of K\"ahler geometry~\cite{shima2007}, since K\"ahler metrics are given by complex Hessians. 
They also appear in affine differential geometry and in information geometry~\cite{Amari, Amari2, PardoGuerra2026}. 
In this way, Hessian geometry connects different areas. 

Specifically, we consider the function $J$ in two coordinate systems used in previous papers~\cite{WZ,WZE}: the original variables
$x$ and the logarithmic variables $t=\log x$. Although these are related by a
smooth change in variables, the Hessian construction depends on the choice of
affine structure, and therefore the induced geometries are not equivalent. In logarithmic coordinates, the function $J$ takes the {form} %MDPI: Please double check all equations and ensure no duplicate equations in the whole paper %MZ We have checked all equations carefully. There are no duplicate equations
\[
J(t)=\cosh\!\left(\sum_{i=1}^n \alpha_i t_i\right)-1.
\]

{Hence} %MDPI: Please confirm whether the no-indent format of paragraph beginning with capital letter should be retained or removed. Please check all similar paragraphs in the main text. %MZ Please retain the standard MDPI formatting style.
 $J$ depends only on the scalar quantity $S(t)=\sum_{i=1}^n \alpha_i t_i$. 
It follows that the Hessian with respect to the $t$-coordinates has rank one at every point. 
The induced geometry is degenerate, with a distinguished direction given by $\alpha$ 
and an integrable $(n-1)$-dimensional null distribution. As a consequence, the associated Hessian structure in logarithmic coordinates effectively reduces to a one-dimensional geometry, even though the ambient space is $n$-dimensional.
On the other hand, when $J$ is expressed in the original $x$-coordinates,
the corresponding Hessian matrix is generically nondegenerate and defines a
pseudo-Riemannian metric away from a singular set. Thus, the same function
gives rise to two qualitatively different geometric structures.

 We further
study the corresponding affine and Levi-Civita geodesics, with particular
attention to the distinguished direction  and to the behavior near the
degeneracy locus. The Hessian manifold defined with logarithmic coordinates is geodesically complete for the affine geodesics, but due to the non-invertibility of the metric, an additional Levi-Civita connection is not defined. Conversely, the Hessian manifold defined with the original coordinates is geodesically incomplete for the affine geodesics (due to a coordinate domain restriction) and for Levi-Civita geodesics (due to the presence of curvature singularities of that connection).

This paper is organized as follows. In Section \ref{define_multiDcost}, we extend the cost function to multiple variables. Next, we examine the construction of Hessian metrics and their properties for an arbitrary dimension in Section \ref{General_geo}. We then examine the properties of the 2-dimensional versions of these manifolds and the trajectories on them to further illustrate their behavior in Section \ref{2D_geo}. In Section \ref{application}, we interpret our results from the perspective of information geometry, relating the cost function to Bregman divergences and Fisher--Rao metrics.
We give our conclusions in Section \ref{conclusion}.

\section{Canonical Reciprocal Cost and Its \boldmath{$n$}-Dimensional Extension}\label{define_multiDcost}
We start with the one-dimensional canonical reciprocal cost function\begin{equation}\label{J-1D}
J(x):=\frac12\left(x+x^{-1}\right)-1,
\qquad x>0.
\end{equation}
This function is our main model. It satisfies the basic properties
\begin{itemize}
\item[(i)] $J(1)=0$,
\item[(ii)] $J(x)=J(x^{-1})$,
\item[(iii)] $J(x)\ge 0$ for all $x>0$.
\end{itemize}
In logarithmic coordinates $t=\log x$, one has
\begin{equation}\label{J-cosh}
J(e^t)=\cosh(t)-1.
\end{equation} 
The importance of this function comes from the following theorem. {We denote by $\mathbb{R}^{+}$ the set of positive real numbers.}

\begin{Theorem}[\cite{WZ}]
Let $F:\mathbb{R}^{+}\to\mathbb{R}$ satisfy composition law 
\begin{equation}\label{comp-law}
F(xy)+F\!\left(\frac{x}{y}\right)
=
2F(x)F(y)+2F(x)+2F(y).
\end{equation}
If the unit log-curvature
\begin{equation}\label{unit-curvature}
\lim_{t\to 0}\frac{2F(e^t)}{t^2}=1
\end{equation}
is satisfied, then
\[
F(x)=\frac12\left(x+x^{-1}\right)-1=J(x).
\]
\end{Theorem}
This theorem shows that the canonical reciprocal cost \eqref{J-1D} is unique under the composition law and the unit log-curvature normalization. The relation \eqref{comp-law} may be viewed as a multiplicative d'Alembert-type functional equation on the positive reals~\cite{WZE,aczel1966,aczeldhombres1989}. The logarithmic reparametrization transforms it into an additive one-dimensional variable.  

We stress that uniqueness in one dimension does not extend to higher dimensions. 
Some additional assumptions are needed. 
Therefore, it is natural to ask for a suitable multidimensional extension, 
and this question motivates the rest of the paper.

Let us consider the $n$-dimensional function on $(\mathbb{R}^{+})^n$ given by
\begin{equation}\label{J-n}
J(x_1,\dots,x_n)
=
\frac12\left(
\prod_{i=1}^n x_i^{\alpha_i}
+
\prod_{i=1}^n x_i^{-\alpha_i}
\right)-1,
\qquad x_i>0,
\end{equation}
where $\alpha_1,\dots,\alpha_n\in\mathbb{R}$. Let
\begin{equation}\label{RR}
R(x_1,\dots,x_n)
=
\prod_{i=1}^n x_i^{\alpha_i}.
    \end{equation}
Then \eqref{J-n} can be written in the form
\[
J(x_1,\dots,x_n)
=
\frac12\left(R+R^{-1}\right)-1.
\]
Thus, all dependence on the coordinates $x_i$ is through the single scalar
$R(x)$, where {$x=(x_1,\dots,x_n)$.}
Hence, the geometry is essentially one-dimensional.

If we assume permutation symmetry, that is,
\[
J(x_{\sigma(1)},\dots,x_{\sigma(n)})
=
J(x_1,\dots,x_n)
\qquad
\text{for all }\sigma\in S_n,
\]
then all variables are equally weighted. More precisely,

\begin{Lemma}
If the function \eqref{J-n} is permutation symmetric, $\alpha=(\alpha_1,\dots,\alpha_n)\neq 0$, and $n\ge3$ then
\(
\alpha_1=\alpha_2=\cdots=\alpha_n.
\) 
If $n=2$, $\alpha_1=\pm\alpha_2$ allows for permutation symmetry.
\end{Lemma}

\begin{proof}
Permutation symmetry {assumption} gives
\[
\prod_{i=1}^n x_{\sigma(i)}^{\alpha_i}
+
\prod_{i=1}^n x_{\sigma(i)}^{-\alpha_i}
=
\prod_{i=1}^n x_i^{\alpha_i}
+
\prod_{i=1}^n x_i^{-\alpha_i}
\]
for every $\sigma\in S_n$ and all $x_i>0$. Hence
\[
\prod_{i=1}^n x_{\sigma(i)}^{\alpha_i}
=
\prod_{i=1}^n x_i^{\alpha_i}
\quad \text{or} \quad
\prod_{i=1}^n x_{\sigma(i)}^{\alpha_i}
=
\prod_{i=1}^n x_i^{-\alpha_i}.
\] 
 Suppose $n\ge3$ and consider a permutation $\sigma$ that only swaps two elements $i\neq j$. Taking $x_j=e$ and $x_i=1$, we obtain
\begin{equation}
    e^{\alpha_j}\cdot1^{\alpha_i}\cdot\prod_{k\neq i,j}^n x_{k}^{\alpha_k}= e^{\pm\alpha_i}\cdot1^{\pm\alpha_j}\prod_{k\neq i,j}^n x_{k}^{\pm\alpha_k}
\end{equation}
Since the values of $x$ are nonzero, {we obtain
\begin{equation}
    e^{\alpha_j\mp\alpha_i}= \prod_{k\neq i,j}^n x_{k}^{\pm\alpha_k-\alpha_k}.
\end{equation} 
} The left-hand side is constant, while the right-hand side depends on the variables $x_k$. 
Since this identity holds for all $x_k>0$, all exponents must vanish. Thus
\[
\pm\alpha_k-\alpha_k=0 \ (k\neq i,j), \qquad \alpha_j\mp\alpha_i=0.
\]
The minus sign implies $\alpha_k=0$ for all $k\neq i,j$, impossible for $n\ge3$ and $\alpha\ne0$. Hence, the sign is $+$ and $\alpha_i=\alpha_j$.
Since $i,j$ are arbitrary, all exponents are equal.

If $n=2$, the condition becomes
\begin{equation}
x_1^{\alpha_1}x_2^{\alpha_2}
=
x_2^{\pm\alpha_1}x_1^{\pm\alpha_2}.
\end{equation}
That is,
\begin{equation}
x_1^{\alpha_1\mp\alpha_2}
=
x_2^{\pm\alpha_1-\alpha_2}.
\end{equation}
This is independent of $x_1,x_2$ if and only if $\alpha_1=\pm\alpha_2$.
\end{proof}
Therefore, permutation symmetry is determined by choosing $a\in\mathbb{R}$ such that
\[
\alpha_i=a,
\qquad i=1,\dots,n.
\]
In this case, the cost function \eqref{J-n} becomes
\[
J(x_1,\dots,x_n)
=
\frac12\left(
\left(\prod_{i=1}^n x_i\right)^a
+
\left(\prod_{i=1}^n x_i\right)^{-a}
\right)-1.
\]{Suppose that the permutation symmetric multidimensional cost function reduces to the one-dimensional case \eqref{J-1D} when all variables are equal, {i.e.:} 
\[
J(x)=J(x,\dots,x).
\]
Because $J(x)=\tfrac12(x+x^{-1})-1$, then from equality above we get
\[
\cosh(an\log x)-1=\cosh(\log x)-1 \quad \text{for all } x>0,
\]
and we obtain $an=\pm 1$, hence $a=\pm \tfrac{1}{n}$.
}
Since $J(R)=J(R^{-1})$, both choices lead to the same cost function. Without loss of generality, we take $a=\frac1n$.
Therefore, the natural choice under the conditions of permutation symmetry and dimensional reduction is
\(
a=\frac1n.
\)
This~gives
\begin{equation}\label{J-n-canonical}
J(x_1,\dots,x_n)
=
\frac12\left(
\left(\prod_{i=1}^n x_i\right)^{1/n}
+
\left(\prod_{i=1}^n x_i\right)^{-1/n}
\right)-1.
\end{equation}
In the rest of the paper, we work with a general $\alpha$, while $\alpha_i=\frac{1}{n}$ is used as a special case. In particular, $\alpha_i=\frac{1}{n}$ ensures the cost function depends only on the geometric mean of the~variables
\[
G(x)=(x_1\cdots x_n)^{1/n}.
\]

\subsection*{Logarithmic Representation}
Let
\[
t_i=\log x_i,
\qquad i=1,\dots,n.
\]
Then
\[
R(x_1,\dots,x_n)
=
\exp\!\left(\sum_{i=1}^n \alpha_i t_i\right),
\]
and the cost function takes the form
\begin{equation}\label{J-log-form}
J(x_1,\dots,x_n)
=
\cosh\!\left(\sum_{i=1}^n \alpha_i t_i\right)-1.
\end{equation}
Thus, the $n$-dimensional function depends only on one linear combination of the logarithmic coordinates. 
This shows that the Hessian in logarithmic coordinates has rank one, independently of the ambient dimension.

\textls[-15]{Recall that under the assumption $\alpha_1=\cdots=\alpha_n=a$, the dimensional \mbox{reduction~condition}}
\[
J(x,\dots,x)=J(x)
\]
implies $a=\frac{1}{n}$, and hence $\alpha_i=\frac{1}{n}$. In this sense, the choice $\alpha_i=\frac{1}{n}$ is canonical.  In this case,  we obtain
\[
J(x_1,\dots,x_n)
=
\cosh\!\left(\frac1n\sum_{i=1}^n t_i\right)-1.
\]
Analogously to the expression in $x$-coordinates \eqref{J-n-canonical}, where the cost depends on the geometric mean, in $t$-coordinates it depends on the arithmetic mean.

\section{Geometry of General Multidimensional Cost}\label{General_geo}

The logarithmic form of the multidimensional reciprocal cost is
\begin{equation}\label{loge}
J(t)=\cosh(S(t))-1,
\qquad
S(t)=\sum_{i=1}^n \alpha_i t_i,\quad t_i=\log x_i,
\end{equation}
where $t=(t_1,\dots,t_n)\in\mathbb{R}^n$ and $\alpha=(\alpha_1,\dots,\alpha_n)^T\in\mathbb{R}^n$ is fixed. Since
\[
\frac{\partial J}{\partial t_i}
=
\sinh(S(t))\,\alpha_i,
\]
we obtain
\[
\frac{\partial^2 J}{\partial t_i\partial t_j}
=
\cosh(S(t))\,\alpha_i\alpha_j,
\qquad i,j=1,\dots,n.
\]
{Hence,} the Hessian matrix of $J$ is
\[
\nabla^2 J(t)=\cosh(S(t))\,\alpha\alpha^T.
\]
Since $\alpha\neq 0$ and $\cosh(S(t))>0$, we have
\[
\nabla^2 J(t)\alpha=\cosh(S(t))(\alpha^T\alpha)\alpha\neq 0,
\]
and therefore
\[
\operatorname{rank}(\nabla^2 J(t))=1.
\]
Thus, the Hessian is degenerate, and its image is  the one-dimensional subspace generated by $\alpha$.

{\begin{Remark}
If $J(t)=f(\alpha\cdot t)$ and $f''\neq 0$, then the Hessian $\nabla^2 J(t)$ has rank one. Indeed,
\[
\nabla^2 J(t)=f''(\alpha\cdot t)\,\alpha\alpha^T,
\]
which is a matrix of rank one. 

In particular, the rank-one property follows from the fact that $J$ depends on a linear form $S(t)=\alpha\cdot t$. The reciprocal cost is a particular example with $f(S(t))=\cosh(S(t))-1$, 
which gives the  factor $\cosh(S(t))$ in the Hessian.
\end{Remark}}

 %To illustrate the multidimensional construction, we consider the following example.  
%We combine functions of $(r,s)$ linearly so that the cost function depends on a single scalar quantity. 
%Trigonometric components corresponding to $r$, $s$, $r+s$ and $r-s$ represent basic harmonic modes of the variables.

%\begin{Example}
%Let $(r,s)\in\Omega\subset\mathbb{R}^2$ and define
%\[
%\Phi(r,s)=
%\bigl(
%\cos r,\ \sin r,\ \cos s,\ \sin s,\ 
%\cos(r+s),\ \sin(r+s),\ 
%\cos(r-s),\ \sin(r-s)
%\bigr).
%\]
%Denote its components by $\phi_i(r,s)$, $i=1,\dots,8$. 
%For $a\in\mathbb{R}^8$ define
%\[
%S_8(r,s;a)=\frac{1}{\sqrt{8}}\sum_{i=1}^8 a_i\,\phi_i(r,s),
%\qquad
%J_8=\cosh(S_8)-1.
%\]
%Then $J_8$ depends only on the scalar quantity $S_8$, showing that even in higher-dimensional ambient spaces the induced geometry is effectively one-dimensional.
%\end{Example}

\subsection{Interpreting the Hessian as a Metric in Logarithmic Coordinates}

We consider the $n$-dimensional reciprocal cost  in logarithmic coordinates (\ref{loge}). Assume $\alpha=(\alpha_1,\dots,\alpha_n)\neq 0$.  The Hessian of $J$ defines a symmetric metric tensor
\begin{equation}\label{tmet}
g_{ij}(t)
:=
\frac{\partial^2 J}{\partial t_i\partial t_j}
=
\alpha_i\alpha_j\cosh(S(t)). 
\end{equation}
The metric tensor (\ref{tmet}) is degenerate and
\[
\operatorname{rank}(g_{ij}(t))=1
\qquad
\text{for all }t\in\mathbb{R}^n .
\]
Hence, the metric is not invertible, and its kernel has dimension $n-1$. 
The quadratic form associated with $g_{ij}$ is
\begin{equation}\label{gprod}
g(v,v)=\cosh(S(t))(\alpha\cdot v)^2.
\end{equation}
The metric measures variations only in the direction of $\alpha$,
while vectors orthogonal to $\alpha$ lie in the
kernel of $g$.
So, the radical distribution
\begin{equation}\label{radicaldistribution}
\mathcal{R}_\alpha := \{v \in \mathbb{R}^n : \alpha \cdot v = 0 \}
    \end{equation}
is this kernel. It is integrable, and its leaves are the affine hyperplanes
\[
\alpha \cdot t = \mathrm{const}.
\]
Indeed, along any curve $t(s) = t(0) + s v$ for a curve parameter $s$ with $v \in \mathcal{R}_\alpha$, we have
\[
\alpha \cdot t(s) = \alpha \cdot t(0),
\]
so the flow remains in the same leaf.

Thus, the geometry in logarithmic coordinates consists of one distinguished direction given by $\alpha$ while the remaining $(n-1)$ directions form a null foliation transverse to it.

 Degenerate metric structures appear in several areas of
mathematical physics. In particular, non-invertible
metrics are studied in Carrollian geometry~\cite{Carrollian}
and in the intrinsic geometry of null hypersurfaces in general relativity (see, e.g.,~\cite{BI1991} and chapter 3 of~\cite{Poisson2009}). 
In Carrollian geometry, the metric typically has {co-rank} one and is
accompanied by a distinguished vector field generating its kernel. Likewise, null hypersurfaces in GR typically are described using an intrinsic metric of signature $(0,+,+)$.

In the present case, the situation is different. The metric
\eqref{tmet} has rank one and therefore {co-rank} $n-1$.
Hence, for $n>2$, it does not define a Carrollian structure. Only in dimension $n=2$ our Hessian metric
has the correct algebraic type for the Carrollian framework~\cite{Carrollian}. {Another area worth mentioning, where a strictly rank $1$ metric is used, is the Newton-Cartan theory (see, e.g.,~\cite{NewtonCartan}). Here, a $4$ dimensional manifold is endowed with a rank $1$ temporal metric of $t_{\mu\nu}=\tau_\mu \tau_\nu$ induced by a one form $\tau_\mu$, together with a spatial co-metric $h^{\mu\nu}$ of rank $3$  such that $h^{\mu\nu}\tau_\mu=0$, and a connection defined such that both $\nabla_\mu\tau_\nu=0$ and $\nabla_\mu h^{\mu\nu}=0$. These examples show that degenerate metrics of low rank naturally arise in mathematical physics. However, the structure considered here is different, since the degeneracy is of {co-rank} $n-1$,  in contrast to the {co-rank} one structures appearing in Carrollian and related geometries, and is induced by a cost potential. This suggests a possible interpretation in terms of effectively one-dimensional dynamics embedded in a higher-dimensional space.}

{
In logarithmic coordinates, the function $J$ depends only on  the linear form $S=\alpha\cdot t$. The Hessian has rank one and measures variations only in the direction of $\alpha$, while the $(n-1)$-dimensional distribution defined by $\alpha\cdot v=0$ lies in its kernel.

In the original $x$-coordinates, the transformation $t=\log x$ is not affine. Therefore, the Hessian construction involves additional terms, and the metric becomes generically nondegenerate outside of singular hypersurfaces.}
 
 \subsection{{Curves in the Kernel of $g$}}

The following is a consequence of the definition of the radical distribution \eqref{radicaldistribution} and describes curves along which the cost function remains constant.

Let $\gamma(\tau)$ be a curve with velocity $\dot{\gamma}=\frac{d\gamma}{d\tau}$ such that
\begin{equation}
g(\dot{\gamma},\dot{\gamma})=\cosh({S(\gamma(\tau))})\,(\alpha\cdot \dot{\gamma})^2=0.\label{nullvelocity}
\end{equation}
Since {$\cosh(S(\gamma(\tau)))>0$}, it follows that $\alpha\cdot \dot{\gamma}=0$.

Let $\beta^k$ be a set of vectors $\beta^k=(\beta^k_i)$,
$k=1,\dots,n-1$,  spanning the radical distribution~\eqref{radicaldistribution} such that
\begin{equation}\label{ortho_alphabeta}
\sum_{i=1}^n \alpha_i \beta^k_i=0. 
\end{equation} If
$t_0 = \gamma(0)$,
then any such curve can be written as
\begin{equation}
\gamma(\tau)=t_0+\sum_{k=1}^{n-1}\gamma_k(\tau)\,\beta^k,
\qquad
\dot{\gamma}(\tau)=\sum_{k=1}^{n-1}\dot{\gamma}_k(\tau)\,\beta^k.\label{curveparams}
\end{equation}
 Thus $\gamma$ lies in the affine hyperplane
\[
\alpha\cdot t=\alpha\cdot t_0,
\]
\textls[15]{and the components $\gamma_k(\tau)$ in \eqref{curveparams}  are arbitrary, since \eqref{ortho_alphabeta} ensures that \eqref{nullvelocity} is satisfied.~Moreover,}
\[
\frac{dS}{d\tau}=\alpha\cdot\dot{\gamma}=0,
\]
so {$S(\gamma(\tau))$} is constant and the cost function {$J(\gamma(\tau))$} remains constant. 
Thus, the {curves satisfying \eqref{nullvelocity}} are integral curves of $\mathcal{R}_\alpha$ and lie in hypersurfaces $J=\text{const}$.

\subsection{Interpreting the Hessian as a Metric in $x$-Coordinates}

{Now we compute} the Hessian with respect to the original coordinates $x_i$.
{Using $R(x)=\prod_{k=1}^n x_k^{\alpha_k}$ from \eqref{RR}, the reciprocal cost is written as}
\[
J(x)=\frac12(R+R^{-1})-1.
\]
Since
\[
\frac{\partial R}{\partial x_i}=\frac{\alpha_i}{x_i}R,
\]
we obtain
\[
\frac{\partial J}{\partial x_i}
=
\frac12(R-R^{-1})\frac{\alpha_i}{x_i}.
\]
Differentiating the last equation gives, for $i\neq j$,
\[
\frac{\partial^2 J}{\partial x_i\partial x_j}
=
\frac12\,\frac{\alpha_i\alpha_j}{x_i x_j}(R+R^{-1}),
\]
while on the diagonal one has
\[
\frac{\partial^2 J}{\partial x_i^2}
=
\frac12\,\frac{1}{x_i^2}
\left(
\alpha_i(\alpha_i-1)R
+
\alpha_i(\alpha_i+1)R^{-1}
\right).
\]Equivalently, by using $u_i=\alpha_i/x_i$ and $D=\mathrm{diag}(\alpha_i/x_i^2)$, we obtain 
\[
\nabla_x^2 J
=\frac12\,(R+R^{-1})\,u u^\top
-\frac12\,(R-R^{-1})\,D.
\] 
This representation shows that the Hessian is the sum of a rank-one
matrix $uu^\top$ and a diagonal matrix $D$. It is therefore invertible in general, but fails to be invertible under certain conditions. For instance, on the hypersurface $R=1$ (equivalently $J=0$), the diagonal term vanishes and therefore
\[
\nabla_x^2 J_{|R=1}=u u^\top,
\]
which is a matrix of rank one (provided $\alpha\neq0$).
In particular, at equilibrium $x_i=1$ \mbox{we obtain}
\[
\nabla_x^2 J(1,\dots,1)=\alpha\alpha^\top.
\]
Assume that $R\neq 1$ and $\alpha_i\neq0$ for all $i$. Then, the diagonal matrix
\[
D=\mathrm{diag}\!\left(\frac{\alpha_i}{x_i^2}\right)
\]
is invertible, {and invertible $A$ is given by}
\[
A:=-\frac12(R-R^{-1})D.
\]
In this case, the Hessian {is} written as
\[
\nabla_x^2J = A+\beta uu^\top,
\qquad
\beta:=\frac12(R+R^{-1}),
\qquad
u_i=\frac{\alpha_i}{x_i}.
\]
{In particular, $A$ is invertible when} $R\neq 1$ and all $\alpha_i\neq 0$, and from the determinant rule
\[
\det(A+\beta uu^\top)
=
\det(A)\bigl(1+\beta\,u^\top A^{-1}u\bigr),
\]
we obtain
\[
\det(\nabla_x^2J)
=
\det(A)\bigl(1+\beta\,u^\top A^{-1}u\bigr).
\]
Hence, under the restrictions  $R\neq 1$ and $\alpha_i\neq0$ for all $i$, the Hessian $\nabla_x^2J$ is singular if and only if
\[
1+\beta\,u^\top A^{-1}u=0.
\] This shows that the degeneracy occurs on a hypersurface
in {$(\mathbb{R}^{+})^n$.} 
We summarize the above consideration {in the next Proposition.}

\begin{Proposition}
Let $R(x)=\prod_{k=1}^n x_k^{\alpha_k}$, and assume $R\neq1$ and
$\alpha_i\neq0$ for all $i$. Then the Hessian $\nabla_x^2J$
is invertible except on the locus
\[
1+\beta\,u^\top A^{-1}u=0,
\]
where $A=-\frac12(R-R^{-1})D$, $\beta=\frac12(R+R^{-1})$,
$u_i=\alpha_i/x_i$, and $D=\mathrm{diag}(\alpha_i/x_i^2)$.
\end{Proposition}
A direct computation gives
\[
1+\beta\,u^\top A^{-1}u
=
1-\frac{R+R^{-1}}{R-R^{-1}}\sum_{i=1}^n \alpha_i
=
1-{\coth(S(t))}\sum_{i=1}^n \alpha_i,
\]
where ${S(t)}$ is defined by \eqref{loge}.
Hence, the singular locus is determined by
\[
{\coth(S(t))}=\frac{1}{\sum_{i=1}^n \alpha_i}
\quad
\mbox{or} \quad
{\tanh(S(t))}=\sum_{i=1}^n \alpha_i,
\]
which has solutions if and only if
\[
\left|\sum_{i=1}^n \alpha_i\right|<1.
\]
Thus, the singular locus defines a proper hypersurface (whenever nonempty), and therefore the Hessian is generically nondegenerate.

Finally, the Hessian metric is given by
\begin{equation}
h_{ij}(x)
=
\begin{cases}
\displaystyle
\frac12\,\frac{\alpha_i\alpha_j}{x_i x_j}\,(R+R^{-1}),
& i\neq j,\\[12pt]
\displaystyle
\frac12\,\frac{1}{x_i^2}
\Big(
\alpha_i(\alpha_i-1)R
+
\alpha_i(\alpha_i+1)R^{-1}
\Big),
& i=j.
\end{cases}\label{generalhmet}
\end{equation}
Thus, $h_{ij}$ defines a Hessian metric which is generically nondegenerate, and hence \mbox{pseudo-Riemannian}. 

{For instance, in the case $n=2$, the signature is determined by the sign of $\det(h_{ij})$.   Starting from 
\[
R(x_1,x_2)=x_1^{\alpha_1}x_2^{\alpha_2},
\]
we obtain 
\[
\det(h_{ij})
=
h_{11}h_{22}-h_{12}^2,
\]
where
\[
h_{12}
=
\frac12\,\frac{\alpha_1\alpha_2}{x_1 x_2}(R+R^{-1}),
\]
\[
h_{11}
=
\frac12\,\frac{1}{x_1^2}
\Big(
\alpha_1(\alpha_1-1)R
+
\alpha_1(\alpha_1+1)R^{-1}
\Big),
\]
\[
h_{22}
=
\frac12\,\frac{1}{x_2^2}
\Big(
\alpha_2(\alpha_2-1)R
+
\alpha_2(\alpha_2+1)R^{-1}
\Big).
\]
Substituting into the determinant and simplifying, {we obtain}
\[
\det(h_{ij})
=
-\frac14\,\frac{\alpha_1\alpha_2}{{R^2}x_1^2 x_2^2}
\,(R^2-1)
\Big((\alpha_1+\alpha_2-1)R^2+\alpha_1+\alpha_2+1\Big).
\]
In particular, up to the positive factor $x_1^2 x_2^2$ in the denominator, $\det(h_{ij})$ depends on $(x_1,x_2)$ only through $R$.

In the canonical case $\alpha_1=\alpha_2=\tfrac12$, we have $R=\sqrt{x_1x_2}$, and 
\[
\det(h_{ij})
=
-\frac{1}{8x_1^3x_2^3}\,(x_1x_2-1), \quad x_1>0,\;x_2>0.
\]
Therefore,
\[
\det(h_{ij})<0 \quad \text{for } x_1x_2>1,
\qquad
\det(h_{ij})>0 \quad \text{for } x_1x_2<1.
\]
Thus, the metric is Lorentzian (signature $(1,1)$) on the region $x_1x_2>1$, positive definite on $x_1x_2<1$, and degenerate on the hypersurface $x_1x_2=1$.}

The degeneracy is not intrinsic to $J$, but depends on the chosen affine structure. We summarize the above results in the following theorem {omitting the proof.}
 
 \begin{Theorem}
Let $J$ be reciprocal cost \eqref{J-n} and assume $\alpha=(\alpha_1,\dots,\alpha_n)\neq 0$. Then:

\begin{itemize}
\item[(i)] In logarithmic coordinates, the Hessian satisfies
\[
\operatorname{rank}{(\nabla^2 J(t))}=1
\]
at every point.

\item[(ii)] {In the original $x$-coordinates, the Hessian $\nabla_x^2 J(x)$  is generically nondegenerate. More precisely, for fixed parameters $\alpha_i\neq 0$, it is invertible on the open set where $R\neq 1$ and
\[
1+\beta\,u^T A^{-1}u\neq 0.
\]

The singular locus determined by the condition
\[
1+\beta\,u^T A^{-1}u=0,
\]
which is equivalent to
\[
\tanh(S(x))=\sum_{i=1}^n \alpha_i,
\qquad
S(x)=\sum_{i=1}^n \alpha_i \log x_i,
\]
 exists if and only if
\[
\left|\sum_{i=1}^n \alpha_i\right|<1.
\]
}
\end{itemize}
\end{Theorem}

 \subsection{Projective Equivalence}

{In the subsection, we} analyze the projective relation between the affine connections of $M_t$ and $M_x$. {
We consider the Hessian manifold $M=(\mathbb{R}^+)^n$.

The manifold $M_t$ corresponds to the flat affine connection for which the logarithmic coordinates $t_i$ are affine, and the metric $g$ in these coordinates is given by \eqref{tmet}. 

The manifold $M_x$ corresponds to the flat affine connection for which the original coordinates $x_i$ are affine, and the metric $h$ is given by \eqref{generalhmet}.} Geodesic mapping and its generalizations were investigated by many authors, for example J. Mike\v s  in~\cite{Mikes2019}, N. S. Sinyukov in~\cite{Sinyukov1979}, L.S. Velimirovi\' c and M.S. Stankovi\' c in~\cite{MS} and many others~\cite{Zlatanovic2012}.

\begin{Proposition} {Let $M=(\mathbb{R}^+)^n$ be the  manifold,} and let $M_t$ and $M_x$ denote the affine structures corresponding to the coordinates $t_i=\log x_i$ and $x_i$, respectively.

For $n\ge 2$, the affine connections of $M_t$ and $M_x$ are not projectively equivalent.
\end{Proposition}

\begin{proof}
In the $x$-coordinates, the affine connection of $M_x$ is trivial:
\[
{}_x\Gamma^i_{jk}=0.
\]
On the other hand, the affine connection of $M_t$, expressed in the $x$-coordinates,
has the nonzero components
\[
{}_t\Gamma^i_{ii}=-\frac{1}{x_i},
\qquad i=1,\dots,n,
\]
while all other components vanish.

Recall that two torsion-free connections are projectively equivalent if and only if
there exists a $1$-form $\psi=\psi_k\,dx^k$ such that
\[
{}_t\Gamma^i_{jk}-{}_x\Gamma^i_{jk}
=
\delta^i_j\psi_k+\delta^i_k\psi_j.
\]
Since ${}_x\Gamma^i_{jk}=0$, this becomes
\[
{}_t\Gamma^i_{jk}
=
\delta^i_j\psi_k+\delta^i_k\psi_j.
\]
Taking $i=j=k$, we obtain
\[
-\frac{1}{x_i}={}_t\Gamma^i_{ii}=2\psi_i,
\]
hence
\[
\psi_i=-\frac{1}{2x_i}
\quad \text{{for all} }\; i=1,\dots,n.
\]
{Fix an index $i\in\{1,\dots,n\}$ and choose $l\neq i$.} Since the mixed components of ${}_t\Gamma$ vanish, taking $j=i$ and $k=l$ we obtain
\[
0={}_t\Gamma^i_{il}
=
\delta^i_i\psi_l+\delta^i_l\psi_i
=
\psi_l.
\]
Thus,
\[
\psi_l=0,
\]
which contradicts
\[
\psi_l=-\frac{1}{2x_l}.
\]
Therefore, for $n\ge 2$, the two affine connections are not projectively equivalent.
\end{proof}

%\subsection{Conformal equivalence}

%Let $M_t$ and $M_x$ be the manifolds determined by the coordinates $t_i=\log x_i$ and $x_i$, respectively, and let $g_{ij}$ and $h_{ij}$ be the corresponding Hessian metrics. 

%For $n\ge 2$, the metrics $g$ and $h$ are not conformally equivalent, since $g$ has rank one while $h$ is generically of full rank on an open set. For $n=1$, the metrics are trivially conformally equivalent.

 \section{Comparative Analysis Between the Two-Dimensional \boldmath{$M_x$} and \mbox{\boldmath{$M_t$} Manifolds}}\label{2D_geo}

 In this section, we consider the case $n=2${,}  use notation $(s,t)=(t_1,t_2)$, $(x,y)=(x_1,x_2)${,} and $(a,b)=(\alpha_1,\alpha_2)$. 
The purpose is to make the general structure explicit in dimension $n=2$, where singular sets, curvature, and geodesics can be written explicitly.
 
\subsection{The Two-Dimensional $M_x$} 
According to the notation aforementioned above, we have
\begin{equation}
    J(x,y)=\frac{1}{2}(x^a y^b+x^{-a}y^{-b})-1,
\end{equation}
and the Hessian of $J$ with respect to the $x,y$ coordinates is
\begin{align}
H=\left(
\begin{array}{cc}
 \frac{a \left((a+1) x^{-a} y^{-b}+(a-1) x^a y^b\right)}{2 x^2}
   & \frac{a b \left(x^{-a} y^{-b}+x^a y^b\right)}{2 x y} \\
 \frac{a b \left(x^{-a} y^{-b}+x^a y^b\right)}{2 x y} & \frac{b
   \left((b+1) x^{-a} y^{-b}+(b-1) x^a y^b\right)}{2 y^2} \\
\end{array}
\right)\label{generalHx2D}
\end{align}
in general. The zero-cost condition for nonzero $a,b$ can be written $y=x^{-a/b}$, so  on the zero-cost hypersurface, we get
\begin{align}
 H=   \left(
\begin{array}{cc}
 \frac{a^2}{x^2} & a b x^{\frac{a}{b}-1} \\
 a b x^{\frac{a}{b}-1} & b^2 x^{\frac{2 a}{b}} \\
\end{array}
\right).\label{J0Hx2D}
\end{align}

At the point $x=y=1$, we get 
\begin{align}
 H=   \left(
\begin{array}{cc}
 a^2 & a b  \\
 a b  & b^2  \\
\end{array}
\right).\label{xy1Hx2D}
\end{align}

The determinant of \eqref{J0Hx2D} is zero, so the matrix is not invertible on the zero-cost hypersurface. 
Since the point $x=y=1$ belongs to the zero-cost hypersurface, the matrix \eqref{xy1Hx2D} is also not invertible.

On the other hand, the determinant of \eqref{generalHx2D} is
\[
-\frac{1}{4} a b x^{-2 (a+1)} y^{-2 (b+1)} \left(x^{2 a} y^{2
   b}-1\right) \left(a x^{2 a} y^{2 b}+b x^{2 a} y^{2 b}-x^{2
   a} y^{2 b}+a+b+1\right),
\]
which is not identically zero. Therefore, away from the zero-cost hypersurface and the additional singular locus, $H$ is invertible and defines a metric.

\subsection{The Two-Dimensional $M_t$}

For comparison, let
\(
s=\log x, t=\log y.
\)
Then,
\begin{align}
    J(s,t)=\cosh(as+bt)-1.
\end{align}
Hence,
\begin{align}
  \partial_i \partial_j J
 =   \left(
\begin{array}{cc}
 a^2\cosh(as+bt) & ab\cosh(as+bt)  \\
 ab\cosh(as+bt)  & b^2\cosh(as+bt)
\end{array}
\right)\label{2dT}
\end{align}
Matching the general $n$-dimensional case in logarithmic coordinates, the Hessian matrix of $J$ with respect to $(s,t)$ coordinates  \eqref{2dT} has rank one:
\[
g_{ij}=\cosh(as+bt)\,\begin{pmatrix}a\\ b\end{pmatrix}
\begin{pmatrix}a & b\end{pmatrix}.
\]
Hence, its kernel is one-dimensional and generated by the vector $(b,-a)$. Thus, in dimension two, the Hessian has rank one, and its kernel is one-dimensional.

\subsection{Two Affine Structures}

Let $X=(x,y)$ and $T=(s,t)$. We use hatted indices for $(x,y)$ coordinates.  Let us denote by $M_t$ the Hessian manifold given by the affine coordinates $(s,t)$, with metric $g_{ij}$ from \eqref{2dT}, and by $M_x$ as the Hessian manifold given by the affine coordinates $(x,y)$, with metric $h_{\hat{i}\hat{j}}$ from \eqref{generalHx2D}.

Although the variables are related by $x=e^s$, $y=e^t$, the Hessian construction depends on the chosen affine structure.  Therefore, $g$ and $h$ are not the same Hessian tensor written in different coordinates, but arise from two different flat connections. The Jacobian factors are
\begin{align}
&\frac{\partial x}{\partial s}=x,\qquad \frac{\partial x}{\partial t}=0,\qquad \frac{\partial y}{\partial s}=0,\qquad \frac{\partial y}{\partial t}=y,\\
&\frac{\partial s}{\partial x}=\frac1x,\qquad \frac{\partial t}{\partial x}=0,\qquad \frac{\partial s}{\partial y}=0,\qquad \frac{\partial t}{\partial y}=\frac1y.
\end{align}
Transforming $g$ to $(x,y)$, we get
\begin{align}
g_{\hat{i}\hat{j}}
=\frac{\partial T^i}{\partial X^{\hat{i}}}\frac{\partial T^j}{\partial X^{\hat{j}}}g_{ij}
=\left(
\begin{array}{cc}
 \frac{a^2 \left(x^{-a} y^{-b}+x^a y^b\right)}{2 x^2} & \frac{ab \left(x^{-a} y^{-b}+x^a y^b\right)}{2 x y} \\
 \frac{ab \left(x^{-a} y^{-b}+x^a y^b\right)}{2 x y} &
   \frac{b^2 \left(x^{-a} y^{-b}+x^a y^b\right)}{2 y^2}
\end{array}
\right).
\end{align}
This is not equal to $h_{\hat{i}\hat{j}}$. 
Similarly,

\begin{equation}
\begin{aligned}
h_{ij}
&=\frac{\partial X^{\hat{i}}}{\partial T^i}\frac{\partial X^{\hat{j}}}{\partial T^j}h_{\hat{i}\hat{j}}
\\&=\left(
\begin{array}{cc}
 a (a \cosh (a s+b t)-\sinh (a s+b t)) & ab \cosh (a s+b t) \\
 ab \cosh (a s+b t) & b (b \cosh (a s+b t)-\sinh (a s+b t))
\end{array}
\right),\label{hmettcoords}
\end{aligned}\end{equation}
which is not equal to \eqref{2dT}.

\subsection{Transformation of the Connections}

\textls[-15]{The difference $g_{ij}\neq h_{ij}$ arises because the two Hessians use different affine~connections.} 

The formula
\(
g_{ij}=\partial_i\partial_j J
\)
is not invariant under nonlinear coordinate change. The invariant form is
\[
g_{ij}=\nabla_i\nabla_j J,
\]
where $\nabla$ is the flat affine connection of the chosen coordinates. For $M_t$, we have
\[
{}_t\Gamma^i_{jk}=0 \quad \text{in } (s,t),
\]
and for $M_x$
\[
{}_x\Gamma^{\hat{i}}_{\hat{j}\hat{k}}=0 \quad \text{in } (x,y).
\]
Transforming connections, we obtain for $M_t$ in $(x,y)$:
\begin{align}
{}_t\Gamma^x_{xx}=-\frac1x,\qquad {}_t\Gamma^y_{yy}=-\frac1y,
\end{align}
all other components zero.

For $M_x$ in $(s,t)$ {coordinates}:
\begin{align}
{}_x\Gamma^s_{ss}=1,\qquad {}_x\Gamma^t_{tt}=1,
\end{align}
{and} all other components {are} zero.
Using the transformed connection coefficients and covariant {derivatives, we obtain}
\[
{}_t\nabla_{\hat{i}}\,{}_t\nabla_{\hat{j}}J=g_{\hat{i}\hat{j}},
\qquad
{}_x\nabla_i\,{}_x\nabla_jJ=h_{ij}.
\]
This shows that the geometric behavior induced by $J$ depends on the chosen affine structure.

\subsection{Levi-Civita Connection and Geodesics in $x$-Coordinates}

Following the notation of \eqref{generalHx2D}, let $x_1=x$, $x_2=y$, $\alpha_1=a$, $\alpha_2=b$, and define
\[
R(x,y)=x^a y^b,\qquad Z=R^2=x^{2a}y^{2b}.
\]
{The metric $h_{\hat{i}\hat{j}}=\partial_{\hat{i}}\partial_{\hat{j}}J$ is given by \eqref{generalHx2D}, where the hatted indices refer to the original $(x,y)$-coordinates, and}
\[
J(x,y)=\frac12\bigl(R+R^{-1}\bigr)-1.
\]
Introduce the denominator factor
\[
\Delta=(Z-1)\bigl((a+b-1)Z+(a+b+1)\bigr).
\]
Note that $\Delta=0$ on the hypersurface $R=1$ (i.e.\ $Z=1$), and also on the additional locus
\[
(a+b-1)Z+(a+b+1)=0,
\]
where the metric is not invertible. Thus, the formulas below are valid on the set
\[
\mathcal{D}=\{(x,y)\in(0,\infty)^2:\Delta\neq 0,\;a,b\neq 0\}.
\]
Since $Z>0$, the equation
\(
(a+b-1)Z + (a+b+1)=0
\) has a positive solution
if 
\[
 Z=-\frac{a+b+1}{a+b-1}>0.
\]
It follows that
\[
(a+b+1)(a+b-1) < 0,
\]
that is, when $|a+b|<1$.

The non-zero Christoffel symbols of the Levi-Civita connection $\Gamma^{k}_{ij}=\Gamma^{k}_{ji}$ are given by
\begin{align*}
\Gamma^{x}_{xx} &=
\frac{Z^2 a^2+2Z^2 ab-3Z^2 a-2Z^2 b+2Z^2-2Za^2+4Zab-4Z+a^2+2ab+3a+2b+2}{2x\,\Delta},\\
\Gamma^{x}_{xy}&=\Gamma^{x}_{yx} =
-\frac{b\bigl(-Z^2 b+Z^2+4Za-2Zb-b-1\bigr)}{2y\,\Delta},\\
\Gamma^{x}_{yy} &=
-\frac{bx\bigl(Z^2 b-Z^2+6Zb+b+1\bigr)}{2y^2\,\Delta},\\
\Gamma^{y}_{xx} &=
-\frac{ay\bigl(Z^2 a-Z^2+6Za+a+1\bigr)}{2x^2\,\Delta},\\
\Gamma^{y}_{xy}&=\Gamma^{y}_{yx} =
\frac{a\bigl(Z^2 a-Z^2+2Za-4Zb+a+1\bigr)}{2x\,\Delta},\\
\Gamma^{y}_{yy} &=
\frac{2Z^2 ab-2Z^2 a+Z^2 b^2-3Z^2 b+2Z^2+4Zab-2Zb^2-4Z+2ab+2a+b^2+3b+2}{2y\,\Delta}.
\end{align*} 

These coefficients ({The} %MDPI: %MDPI: Footnotes are not allowed in this journal. We moved contents into paragraph. Please confirm. %MZ it is okay
 Christoffel symbols and Ricci scalar here and in Section~\ref{LCgeo} are verified using Wolfram Mathematica.)
 show that the Levi-Civita connection is singular on the hypersurface $\Delta=0$, which includes the zero-cost hypersurface $R=1$.

We now compute the curvature of this connection. Since the manifold is two-dimensional, there is only one independent curvature invariant. {
The Ricci scalar of the Levi-Civita connection of $h$ (see, e.g.,~\cite{doCarmo1992}, {Chapter 4}) is defined by
\(
\mathrm{Ric} = h^{ij} R_{ij}\), where $R_{ij}$ denotes the Ricci tensor. 
In our coordinates, it takes the form}
\[
\mathrm{Ric}
=
\frac{4 (a+b) Z^{3/2} \left((a+b-2)Z +a+b+2\right)}
{(Z-1)^2\left((a+b-1)Z +a+b+1\right)^2}.
\]
This expression shows that the curvature vanishes when $a+b=0$, and diverges on the loci where the denominator vanishes, {in particular, it is at} $Z=1$, that is on the hypersurface $R=1$. See Figure \ref{fig:Ricciscalar}.
\begin{figure}[H]
    %\centering
    \includegraphics[width=0.6\linewidth]{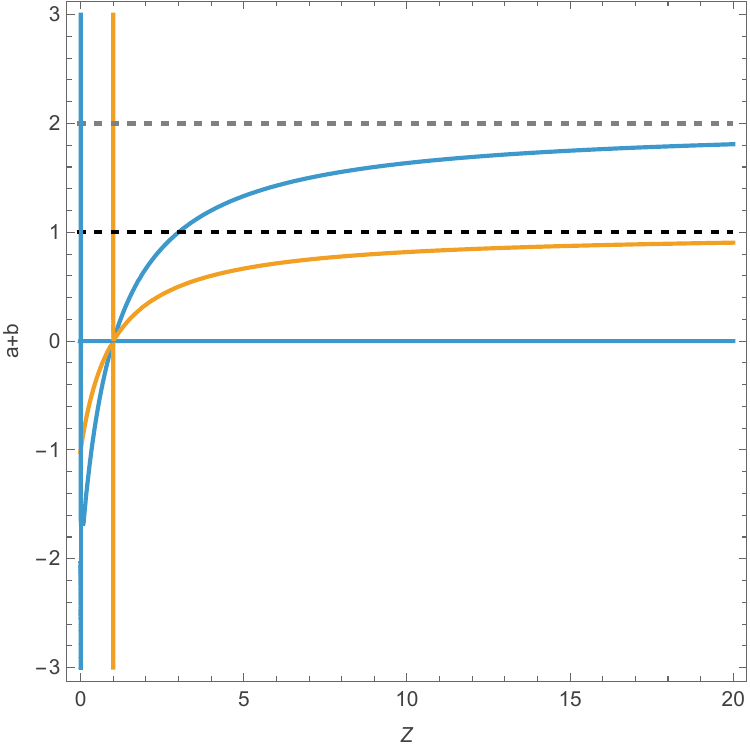}
    \caption{
   Ricci scalar divergences on the orange curves and vanishes on the blue curves. The dashed lines $a+b=2$ and $a+b=1$ represent asymptotes.}
    \label{fig:Ricciscalar}
\end{figure}

Both manifolds $M_t$ and $M_x$, equipped with their natural flat affine connections, have identically zero Riemann curvature. However, when the Levi-Civita connection is introduced on $M_x$, curvature appears.

\textls[-15]{This reflects the fact that the Hessian construction depends on the chosen affine~structure.}

\subsection{Geodesics}

We distinguish three types of curves associated with the two affine structures
introduced above: affine geodesics on $M_t$, affine geodesics on $M_x$, and
Levi-Civita geodesics associated with the metric induced on $M_x$.

\subsubsection{Affine Geodesics on  $M_t$}

On the manifold $M_t$, the affine connection is trivial in the affine coordinates $(s,t)$, {so}
\(
{}_t\Gamma^i_{jk}=0.
\)
Hence, the affine geodesics are straight lines,
\begin{align*}
s(\lambda)=s_0+s_1\lambda,\quad
t(\lambda)=t_0+t_1\lambda.
\end{align*}
Passing to $x$-coordinates via $x=e^s$, $y=e^t$, these become
\begin{align*}
x(\lambda)=x_0 e^{s_1\lambda},\quad
y(\lambda)=y_0 e^{t_1\lambda}.
\end{align*}
These curves are defined for all $\lambda\in\mathbb{R}$ and are the images of affine geodesics on $M_t$ under the transformation $x=e^s$, $y=e^t$.  Higher-dimensional versions of $M_t$ will also have straight line geodesics in logarithmic coordinates with respect to the intrinsic affine connection.

\subsubsection{Levi-Civita Geodesics on  $M_x$}\label{LCgeo}

On the manifold $M_x$, we consider the Levi-Civita connection of the metric
$h_{\hat i\hat j}$ induced by the Hessian structure. The geodesics satisfy
\[
\frac{d^2 x^i}{d\lambda^2} + \Gamma^i_{jk}(x)\,\frac{dx^j}{d\lambda}\frac{dx^k}{d\lambda}=0.
\]

The Levi-Civita connection is defined only on the open set where the metric is nondegenerate,
i.e., where $\Delta\neq 0$. The corresponding geodesic equations become singular on the degeneracy set $\Delta=0$.
 
Explicitly, the equations are
\vspace{-12pt}
\begin{adjustwidth}{-\extralength}{0cm}
\centering %% If there is a figure in wide page, please release command \centering, for Table, ``\textwidth" should be ``\fulllength"
\begin{equation}\label{xgeo}
\begin{aligned}
&\frac{d^2 x}{d\lambda^2}
= -\frac{(a+1)(a+2b+2) - 2Z(a^2 - 2ab + 2) + (a-1)Z^2(a+2b-2)}{2\Delta\, x}
\left(\frac{dx}{d\lambda}\right)^2 \\
&- \frac{b\bigl(2Z(b-2a) + (b-1)Z^2 + b + 1\bigr)}{\Delta\, y}
\frac{dx}{d\lambda}\frac{dy}{d\lambda} + \frac{b x \bigl((b-1)Z^2 + 6bZ + b + 1\bigr)}{2\Delta\, y^2}
\left(\frac{dy}{d\lambda}\right)^2,
\end{aligned}
\end{equation}
\end{adjustwidth}
\vspace{-12pt}
\begin{adjustwidth}{-\extralength}{0cm}
\centering %% If there is a figure in wide page, please release command \centering, for Table, ``\textwidth" should be ``\fulllength"
\begin{equation}\label{ygeo}
\begin{aligned}
&\frac{d^2 y}{d\lambda^2}
= -\frac{(b+1)(2a+b+2) - 2Z(-2ab + b^2 + 2) + (b-1)Z^2(2a+b-2)}{2\Delta\, y}
\left(\frac{dy}{d\lambda}\right)^2 \\
& - \frac{a\bigl(2Z(a-2b) + (a-1)Z^2 + a + 1\bigr)}{\Delta\, x}
\frac{dx}{d\lambda}\frac{dy}{d\lambda}  + \frac{a y \bigl((a-1)Z^2 + 6aZ + a + 1\bigr)}{2\Delta\, x^2}
\left(\frac{dx}{d\lambda}\right)^2,
\end{aligned}
\end{equation} 
\end{adjustwidth}
where $\Delta$ and $Z$ are defined as above.

Note the symmetry under simultaneously swapping $a$ and $b$, and $x$ and $y$. A point on the zero-cost hypersurface can be written as
\begin{equation}
    x=z_0^{1/a}, \quad y=z_0^{-1/b},
\end{equation}
for some $z_0>0$ and nonzero $a,b$. The geodesic equations are singular on the $J=0$ hypersurface, since $\Delta=0$ there (equivalently, $Z=1$).

We consider a formal expansion near $\Delta=0$ and extract only leading-order constraints. Suppose a formal solution is of the form
\begin{align}
    x(\lambda)=z_0^{1/a}+\sum_{n=1}^{\infty} x_n(\lambda-k)^n,\qquad
    y(\lambda)=z_0^{-1/b}+\sum_{n=1}^{\infty} y_n(\lambda-k)^n,
    \label{supposedseries}
\end{align}
{ and exists} near $\lambda=k$, so that the curve extends away from the $J=0$ hypersurface and the partial derivatives
\[
\left.\frac{dx}{d\lambda}\right|_{\lambda=k}=x_1,\qquad
\left.\frac{dy}{d\lambda}\right|_{\lambda=k}=y_1
\]
are finite.

After moving all terms in \eqref{xgeo} and \eqref{ygeo} to the left-hand side and multiplying by $\Delta$, the first equation takes  the form
\[
\Delta\frac{d^2x}{d\lambda^2}
+A\left(\frac{dx}{d\lambda}\right)^2
+B\left(\frac{dy}{d\lambda}\right)^2
+C\frac{dx}{d\lambda}\frac{dy}{d\lambda}=0,
\]
with an analogous expression for the $y$-equation. Substituting \eqref{supposedseries}, and expanding near $\lambda=k$ {give} the constraints
\begin{equation}
    4 b z_0^{-1/a} \left(x_1-y_1 z_0^{\frac{1}{a}+\frac{1}{b}}\right)
   \left(b y_1 z_0^{\frac{1}{a}+\frac{1}{b}}+a x_1\right)=0
\end{equation}
 from \eqref{xgeo}, and
\begin{equation}
    4 a z_0^{-\frac{2}{a}-\frac{1}{b}} \left(y_1
   z_0^{\frac{1}{a}+\frac{1}{b}}-x_1\right) \left(b y_1
   z_0^{\frac{1}{a}+\frac{1}{b}}+a x_1\right)=0
\end{equation}
 from \eqref{ygeo}, respectively, for $a,b\neq 0$. Both equations give the same two possibilities:
\[
x_1 = y_1 z_0^{\frac{1}{a}+\frac{1}{b}}, 
\qquad 
x_1 = -\frac{b}{a} y_1 z_0^{\frac{1}{a}+\frac{1}{b}},
\]
since $z_0>0$. These give the possible tangent directions near the zero-cost hypersurface compatible with obeying \eqref{supposedseries}.
{Since the geodesics are singular on the hypersurface $\Delta=0$, these conditions provide only necessary constraints on the tangent directions. In general, they do not guarantee the existence of a regular solution extending across the singular set. 
Therefore, no  matching conditions arise for Levi-Civita geodesics.}

We also consider the Levi-Civita structure on $M_x$ in $(s,t)$-coordinates, using \eqref{hmettcoords}. In these coordinates, the connection coefficients are
\begin{align}
    \Gamma^s_{ss}&=\frac{a \left(2 b \coth ^2(a s+b t)-\coth (a s+b t)+a\right)}{2 (a+b) \coth (a s+b t)-2},\\
    \Gamma^s_{st}&=\Gamma^s_{ts}=\frac{b \left((b-a) \coth ^2(a s+b t)-\coth (a s+b t)+a\right)}{2 (a+b) \coth (a s+b t)-2},\\
    \Gamma^s_{tt}&=-\frac{b\; \text{csch}^2(a s+b t) (3b-\sinh (2 (a s+b t))+b \cosh (2 (a s+b t)))}{4 (a+b) \coth (a s+b t)-4},\\
    \Gamma^t_{ss}&=-\frac{a\; \text{csch}^2(a s+b t) (3a-\sinh (2 (a s+b t))+a \cosh (2 (a s+b t)))}{4 (a+b) \coth (a s+b t)-4},\\
    \Gamma^t_{ts}&=\Gamma^t_{st}=\frac{a \left((a-b) \coth ^2(a s+b t)-\coth (a s+b t)+b\right)}{2 (a+b) \coth (a s+b t)-2},\\
    \Gamma^t_{tt}&=\frac{b \left(2 a \coth ^2(a s+b t)-\coth (a s+b t)+b\right)}{2 (a+b) \coth (a s+b t)-2}.
\end{align}
The coefficients depend only on the combination
\begin{equation}
    S= as + bt=q,\label{qdefined}
\end{equation}
which is analogous to $Z = x^{2a}y^{2b}$ in the $(x,y)$-coordinates. Define
\begin{align}
    r = -bs + at.\label{rdefined}
\end{align}
Then, $q$ is the linear combination controlling the logarithmic Hessian, while $r$ parametrizes the transverse direction.
{The Ricci scalar of the Levi-Civita connection of $h$ in the $(s,t)$-coordinates is given by}
\begin{equation}\label{ric}
\mathrm{Ric}(q)=
\frac{(a+b)\bigl((a+b)\coth q-2\bigr)\text{csch}^3 q}
{2\bigl((a+b)\coth q-1\bigr)^2}.
\end{equation}
The previous equation shows that the Ricci scalar depends only on $q=as+bt$, and is independent of  $r$. Using the relations
\[
s=\frac{aq-br}{a^2+b^2},\qquad
t=\frac{bq+ar}{a^2+b^2},
\]
the geodesic equations can be written in the $(q,r)$-coordinates as
\begin{equation}\label{qgeoqr}
\begin{aligned}
&2(a^2+b^2)^2 q''\bigl((a+b)-\tanh(q)\bigr)-
ab\big(2(a^2-b^2)q'\, r'
+(a+b)^2\,(r')^2)\\&
+(q')^2\Bigl((a+b)(a^2+b^2)^2\tanh(q)
-\bigl(a^4-a^3b+4a^2b^2-ab^3+b^4\bigr)\Bigr)=0,
\end{aligned}
\end{equation}
{and}
\begin{equation}\label{rgeoqr}
\begin{aligned}
&2(a^2+b^2)^2 r''\bigl((a+b)-\tanh(q)\bigr)
+2(a+b)q'\, r' \Bigl((a^2+b^2)^2\coth(q) 
-(a^3+b^3)\Bigr)\\&
-(a-b)(q')^2\frac{\text{csch}(q)}{2\cosh(q)}
\Bigl( (a^2+b^2)^2(\cosh(2q)+3)-(a^3+b^3)\sinh(2q)\Bigr)\\& 
+ab(a^2-b^2)\,(r')^2=0.
\end{aligned}
\end{equation}
All coefficients depend only on $q$, but the equations are coupled. In general, the condition $r'=0$ is not preserved, except in the symmetric case $a=b$.

We now look for a formal solution near the zero-cost hypersurface, i.e., near $q=0$, in the form
\begin{align}
q(\lambda)=\sum_{n=1}^{\infty} q_n(\lambda-k)^n,\qquad
r(\lambda)=\sum_{n=0}^{\infty} r_n(\lambda-k)^n.
\label{supposedseries2}
\end{align}
The lowest order constraint necessary for such a solution to exist is
\begin{equation}
r_1=\frac{a-b}{a+b}\,q_1.
\end{equation}

\subsubsection{Affine Geodesics on $M_x$}

On $M_x$ manifold, the intrinsic affine connection is trivial in the $(x,y)$-coordinates, i.e.
\(
{}_x\Gamma^i_{jk}=0.
\)
Hence, the affine geodesics are straight {lines:}
\begin{align*}
x(\lambda)=X_1\lambda+X_0,\quad
y(\lambda)=Y_1\lambda+Y_0.
\end{align*}
Passing to logarithmic {coordinates:}
\[
s=\log x,\qquad t=\log y,
\]
these curves satisfy
\begin{align*}
\frac{d^2 s}{d\lambda^2}+\left(\frac{ds}{d\lambda}\right)^2=0,\quad
\frac{d^2 t}{d\lambda^2}+\left(\frac{dt}{d\lambda}\right)^2=0.
\end{align*}
The solutions are given by
\begin{align*}
s(\lambda)=\log(\lambda+S_1)+S_0,\quad
t(\lambda)=\log(\lambda+T_1)+T_0,
\end{align*}
for suitable constants $S_0,S_1,T_0,T_1$, related to $X_0,X_1,Y_0,Y_1$ by
\[
S_0=\log X_1,\qquad S_1=\frac{X_0}{X_1},\qquad
T_0=\log Y_1,\qquad T_1=\frac{Y_0}{Y_1}.
\]
Since $x(\lambda)>0$ and $y(\lambda)>0$, the parameter $\lambda$ is restricted to an interval where
\[
\lambda+S_1>0,\qquad \lambda+T_1>0.
\]
This restriction comes from the domain $(0,\infty)^2$. The affine connection is regular, and the geodesics reach the boundary $x=0$ or $y=0$. For higher-dimensional versions of $M_x$,  geodesics with respect to the intrinsic affine connection  are straight lines in positive ratio coordinates, with a corresponding restriction on the affine parameter.

Thus, the same function $J$ induces three distinct families of curves: the affine geodesics of $M_t$, the affine geodesics of $M_x$, and the Levi-Civita geodesics on $M_x$.

\subsection{Gradient Paths}

We consider the Euclidean gradient flow of the cost
function in the affine logarithmic~coordinates
\begin{equation}
\frac{dt}{d\tau}=\nabla J=\alpha \sinh\!\left(\sum_{i=1}^n \alpha_i t_i\right).\label{gradpath}
\end{equation}
Notice that these paths are defined independently of the Hessian metric or the choice of affine connection, although here the gradient is expressed in $t$-coordinates.
We record both ascent and descent flows: along ascent $J$ is increasing, while along descent $J\to 0$.

In the $2$-dimensional case, writing $(t_1,t_2)=(s,t)$ and $\alpha=(a,b)$, we obtain
\[
\frac{ds}{d\tau}=a \sinh(as+bt), \qquad 
\frac{dt}{d\tau}=b \sinh(as+bt).
\]
Using the combinations $q$ and $r$ from Equations \eqref{qdefined} and \eqref{rdefined}, we  get
\begin{align}
    \frac{dq}{d\tau}&=a\frac{ds}{d\tau}+b\frac{dt}{d\tau}
    =(a^2+b^2)\sinh(q)=|\alpha|^2\sinh(q),\label{dqdt}\\
    \frac{dr}{d\tau}&=-b\frac{ds}{d\tau}+a\frac{dt}{d\tau}
    =(-ba+ab)\sinh(q)=0.
\end{align}From the last two equations, we obtain
\begin{equation}  
\tanh\frac{q(\tau)}{2}=C e^{|\alpha|^2\tau},
\qquad
r(\tau)=C_1,
\end{equation}
for constants $C,C_1$, on the interval where $|Ce^{|\alpha|^2\tau}|<1$. Therefore, the solution for $q$ is
\begin{equation}\label{qgradsolution}
q(\tau)=2\,\mathrm{artanh}\!\left(Ce^{|\alpha|^2\tau}\right)
=\ln\!\left(\frac{1+Ce^{|\alpha|^2\tau}}{1-Ce^{|\alpha|^2\tau}}\right).
\end{equation}

This behavior extends to the $n$-dimensional case. {Using
$S(t)=\sum_{i=1}^n \alpha_i t_i$ from \eqref{loge}, we obtain
$\partial S/\partial t_i=\alpha_i$.  Recall the vectors $\beta^k=(\beta^k_i)$,
$k=1,\dots,n-1$,  spanning the radical distribution \eqref{radicaldistribution}, as in \eqref{ortho_alphabeta} i.e.}
$
\sum_{i=1}^n \alpha_i \beta^k_i=0, 
$
and denote perpendicular directions to $S(t)$ as $r^k$ such that
\(
{\partial r^k}/{\partial t_i}=\beta^k_i.
\)
Then{, abbreviating $S(t(\tau))$ along the gradient paths obeying \eqref{gradpath} {as $S$, we have}
\begin{align}
    \frac{dS}{d\tau}
    &=\sum_{i=1}^n \alpha_i \frac{dt_i}{d\tau}
    =\sum_{i=1}^n \alpha_i \alpha_i \sinh(S)
    =|\alpha|^2 \sinh(S),\label{dsdt}
\end{align}
and
\begin{align}
    \frac{d r^k}{d\tau}
    &=\sum_{i=1}^n\frac{\partial r^k}{\partial t_i}\frac{dt_i}{d\tau}
    =\sum_{i=1}^n\beta^k_i \alpha_i \sinh(S)=0,
\end{align}
by the orthogonality \eqref{ortho_alphabeta}.  Notice that the Equation \eqref{dsdt} coincides with \eqref{dqdt} after replacing $q$ by $S$. Hence, the solution for  $S(\tau)$ follows from \eqref{qgradsolution}.

The cost function $J=\cosh(S)-1$ evolves as
\begin{equation}
    \frac{dJ}{d\tau}
    =\frac{dJ}{dS}\frac{dS}{d\tau}
    =|\alpha|^2 \sinh^2(S)\ge 0,
\end{equation}
so the flow $\dot t=\nabla J$ corresponds to a monotone increase in the cost.

If instead one considers the gradient descent flow
\[
\frac{dt}{d\tau}=-\nabla J=-\alpha \sinh(S),
\]
then
\begin{align*}
     &\frac{dS}{d\tau}
     =-|\alpha|^2 \sinh(S),\\
    & \tanh\frac{S(\tau)}{2}
     =C e^{-|\alpha|^2\tau},\\
     &r^k=C_k,
\end{align*}
for constants $C,C_k$ (on the interval where $|C e^{-|\alpha|^2\tau}|<1$)
and
\begin{equation}
    \frac{dJ}{d\tau}
    =-|\alpha|^2 \sinh^2(S)\le 0.
\end{equation}
{Thus, along the negative gradient flow,  the cost $J$ does not increase strictly with $\tau$, and decreases strictly whenever $S\neq 0$. }

 \subsubsection*{Illustrated Trajectories on $M_x$}

The following figures illustrate several features on $M_x$. The dark curve represents the zero-cost, while the light blue curve, when present, indicates the secondary locus where the Ricci scalar diverges. Levi-Civita geodesics are shown in rainbow colors, where the hue depends on the fractional part of the curve parameter $\lambda$ (in particular, the curve is red when $\lambda$ is an integer). Gray arrows indicate the gradient of $J$, relevant for the gradient flow. The same sets of items are shown in both $xy$ and $qr$ coordinates (see \mbox{Figures \ref{fig:multiplot1} and \ref{fig:multiplot2}}). Notice that these are by no means an exhaustive set for all possible behavior of the Levi-Civita geodesics. Geodesics can be defined with different initial conditions (position and velocity), and different parameter values $\alpha_i$, so they can approach the singular surfaces or travel off to infinity in different ways. Geodesics may approach the singular locus, and in some numerical examples, they can appear to "intersect" it  for certain initial conditions (for example, $a=1/3$, $b=1/3$, $x_0=7$, $y_0=5$, $x_0' = y_0' = -1$). However, this is not a regular intersection. It reflects the breakdown of the geodesic equations near the singular set, where the Christoffel symbols become large. Therefore, the trajectories cannot be extended smoothly across the singular locus.
\begin{figure}[H]
    %\centering
    \includegraphics[width=0.45\linewidth]{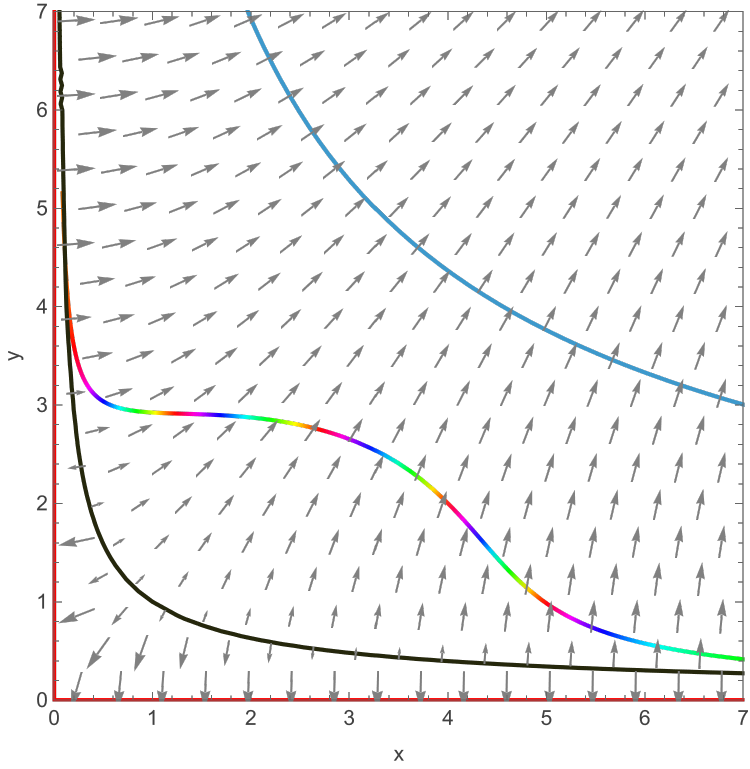}
    \includegraphics[width=0.45\linewidth]{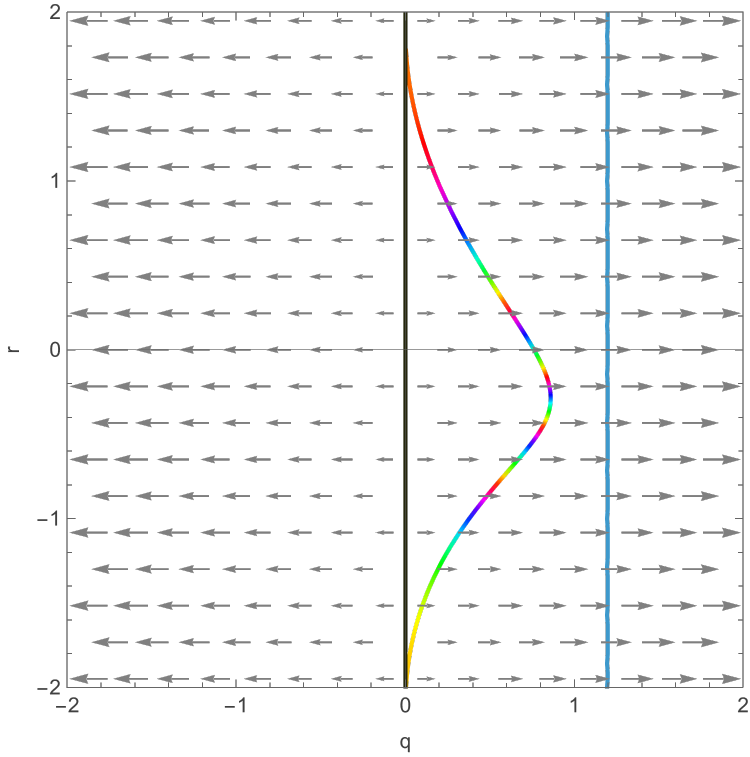}
    \caption{{System} %MDPI: 1. Please confirm whether explanations of the arrow and colored line need to be added to the figure caption.  %MZ it is okay as it stay
 plots for $a=1/3$, $b=1/2$ showing the zero-cost hypersurface (dark), singular locus (blue), gradient of $J$ (gray {arrows}), and a Levi-Civita geodesic with initial data $x(0)=4$, $y(0)=2$, $x'(0)=-1$, $y'(0)=1$, in $xy$ and $qr$ coordinates.}%MDPI: We added arrows here. please confirm %MZ it is okay, thanks
    \label{fig:multiplot1}
\end{figure}
\vspace{-6pt}
\begin{figure}[H]
   % \centering
    \includegraphics[width=0.45\linewidth]{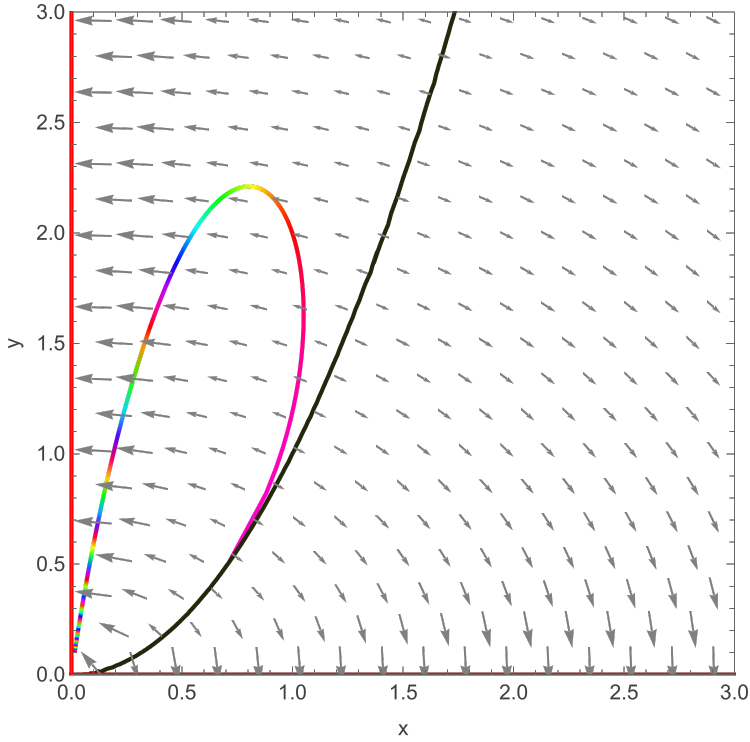}
    \includegraphics[width=0.45\linewidth]{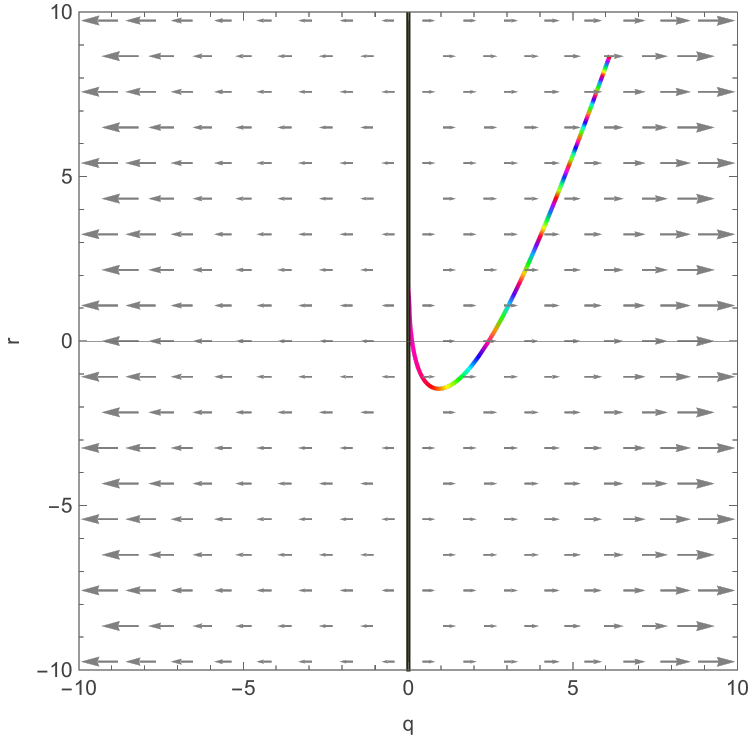}
    \caption {{System} %MDPI: Please confirm whether an explanation of the colored line needs to be added to the figure caption.  %MZ it is okay as it stay
 plots for $a=-2$, $b=1$ showing the zero-cost hypersurface (dark), the gradient of $J$ (gray arrows), and a Levi-Civita geodesic with initial data $x(0)=1$, $y(0)=2$, $x'(0)=-1$, $y'(0)=3$, in both $xy$ and $qr$ coordinates. 
    The color modulation tightens along one segment of the geodesic, indicating a decrease in the magnitude of  the velocity $|\dot{\gamma}(\lambda)|$.
    }
    \label{fig:multiplot2}
\end{figure}

The numerically computed geodesics (obtained in $xy$ coordinates) are substituted into the geodesic equations in $qr$ coordinates, i.e. into the left-hand sides {(LHS)} of \eqref{qgeoqr} and \eqref{rgeoqr}. We define the residual by
\begin{equation}
\text{residual}
=
\bigl|\text{LHS of \eqref{qgeoqr}}\bigr|
+
\bigl|\text{LHS of \eqref{rgeoqr}}\bigr|.
\label{reserr}
\end{equation}
For an exact solution, the residual is zero. In the numerical case, it remains small (typically of order $10^{-10}$), except near singularities where the coefficients become large and errors are amplified. A larger value may also appear near $\lambda=0$ due to the numerical construction of the solution. The residual error for the configurations in Figures \ref{fig:multiplot1} and \ref{fig:multiplot2} are shown in Figure \ref{fig:residual_error}.

\begin{figure}[H]
    %\centering
    \includegraphics[width=0.45\linewidth]{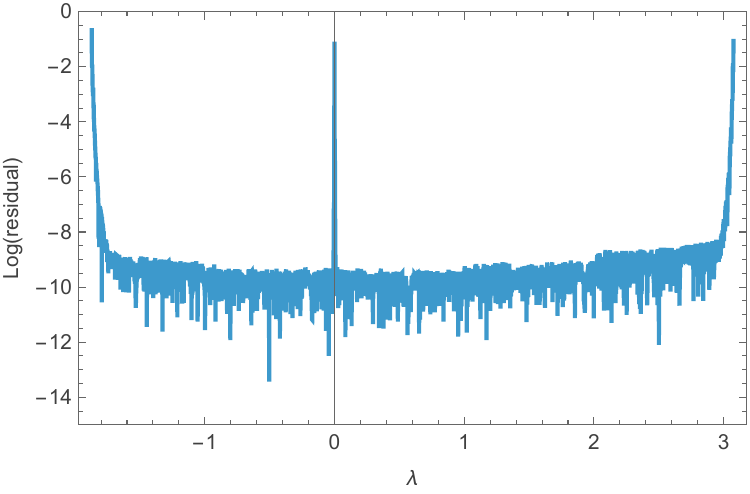}  \includegraphics[width=0.45\linewidth]{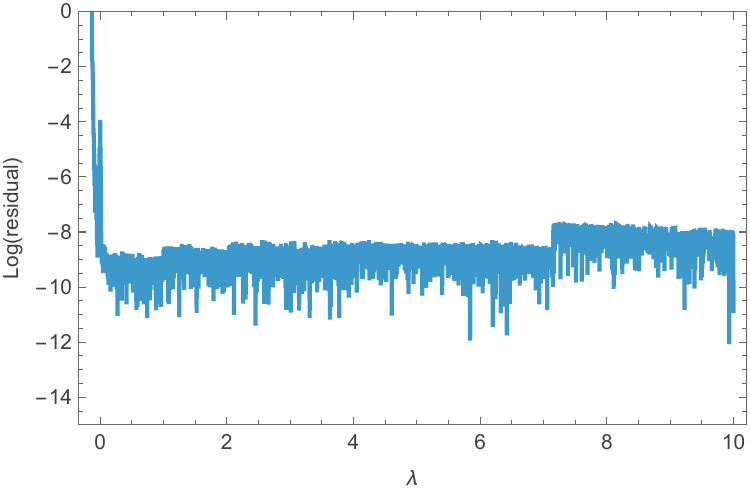}
\caption{Residual error \eqref{reserr} along representative Levi-Civita geodesics.
Left: example from Figure~\ref{fig:multiplot1}, where the error increases near singularities (due to large Christoffel symbols) and shows a local increase at $\lambda=0$ due to the numerical construction.
Right: example from Figure~\ref{fig:multiplot2}, where for positive $\lambda$ the geodesic is sufficiently far from the singularities that the error remains small.
}
    \label{fig:residual_error}
\end{figure}

\section{Application}\label{application}

In this section, we discuss the information-geometric interpretation of the cost function. First, we relate the cost function to the Itakura--Saito divergence. 
For positive scalars~\cite{Itakura1975}, we~have
\[
D_{\mathrm{IS}}(p\|q)=\frac{p}{q}-\log\!\left(\frac{p}{q}\right)-1 .
\]
A direct computation gives
\[
\frac12\bigl(D_{\mathrm{IS}}(1\|x)+D_{\mathrm{IS}}(x\|1)\bigr)
=\tfrac12\left(x+\frac1x-2\right)
=J(x).
\]

Thus, the multidimensional cost is the symmetrized Itakura--Saito divergence applied to $R(x)$ (\ref{RR}), with $R(x)=x$ in the one-dimensional case. This representation carries the Hessian-geometric structure induced by $J$.

The Bregman divergence associated with a convex potential $\phi$ satisfies
\[
D_\phi(x,x+\delta)=\tfrac12\langle \delta,(\nabla^2\phi)(x)\delta\rangle+O(|\delta|^3),
\]
such as in~{\cite{li2021}. }%MDPI: Refs. 17-23 are not cited in the main text. Please add. %MZ I added citations through the manuscript where they naturally fit the discussion. %For several references, I could not find a sufficiently relevant position without unnecessarily expanding the text, so I moved them to comments. If it is preferable to keep them in the bibliography, they may remain as general background references  even without explicit citation in the main text. 
In logarithmic coordinates, $J(t)=\cosh(\alpha\cdot t)-1$ is convex, and the Hessian metric \eqref{tmet} is the second-order term of the Bregman divergence associated with $J$ in arbitrary~dimension. 

 The Bregman interpretation does not extend to the Hessian in $x$-coordinates, since the metric is not positive definite on $(0,\infty)^n$.

The Hessian metric of the cost function $J$ in logarithmic coordinates
\[
g_{ij}(t)=\cosh(\alpha\cdot t)\,\alpha_i\alpha_j
\]
can  be realized as a Fisher--Rao metric of a statistical model~\cite{Amari2}.

Let $S=\alpha\cdot t$. A (not necessarily unique) statistical model which works is the family of normal distributions
\[
p(z;t)=\frac{1}{\sqrt{2\pi}}
\exp\!\left(-\frac{(z-m(S))^2}{2}\right),
\qquad
m(S)=\int_0^S \sqrt{\cosh u}\,du.
\]
Then,
\[
I_{ij}(t)=(m'(S))^2\,\alpha_i\alpha_j.
\]
Since $m'(S)=\sqrt{\cosh S}$, we obtain
\[
I_{ij}(t)=\cosh(\alpha\cdot t)\,\alpha_i\alpha_j=g_{ij}(t).
\]

Therefore, the logarithmic Hessian geometry associated to the cost function admits   an interpretation as a Fisher information geometry of a one-dimensional statistical model embedded in $\mathbb{R}^n$.

\section{Conclusions}\label{conclusion}

{The first contribution of this paper was the extension of the canonical reciprocal cost to multiple dimensions. While the one-dimensional cost is uniquely determined, the multidimensional extension is not unique in general. Starting from a one-variable cost function, we introduced a multidimensional extension parametrized by weights $\alpha$, and showed that permutation symmetry (equal weights), together with the dimensional reduction condition
\(
J(x,\dots,x)=J(x),
\)
leads to the choice $\alpha_i=\frac{1}{n}$, which is canonical in this sense.}

After the multidimensional extension was defined, we analyzed the geometric structures it induces. The main observation is that the same function generates qualitatively
different geometries depending on the chosen affine structure. In logarithmic coordinates,
the Hessian metric has constant rank one, which leads to a degenerate geometry with an
integrable $(n - 1)$-dimensional radical distribution and a direction given by $\alpha$. In contrast,
in the original $x$-coordinates the corresponding Hessian is generically nondegenerate and
defines a pseudo-Riemannian metric away from a singular locus.

We further showed that these differences are reflected at the level of geodesics. The affine
geodesics on both $M_x$ and $M_t$ are straight lines, and in $M_t$ are globally defined,
while in $M_x$ the domain restriction $x_i > 0$  restricts their extendibility. The Levi-Civita
geodesics introduce additional structure, with coefficients depending only on $S = \alpha \cdot t$,
which confirms that the geometry is controlled by a single direction. Therefore, the function $J$ induces three distinct families of curves: affine geodesics of $M_t$, affine geodesics of $M_x$, and Levi-Civita geodesics on $M_x$, depending on the chosen affine or metric structure. 

Independent of the Hessian structure or geodesics, it is natural to examine gradient paths of the cost function. Here, the gradient paths are integral curves of the vector field $\sinh(\alpha\cdot~t)\,\alpha$; hence, the velocities  lie in the one-dimensional distribution generated by $\alpha$. Gradient paths therefore evolve only in the direction $\alpha$.

{
The degeneracy of the Hessian metric has a   geometric meaning. In logarithmic coordinates, the metric has rank one and depends only on the scalar variable
\(
S=\sum_{i=1}^n \alpha_i \log x_i,
\)
so that directions orthogonal to $\alpha$ lie in the kernel of the metric. Thus, the geometry is effectively one-dimensional, determined by the variable $S$, while the remaining directions do not contribute to the metric.}

{Our further work will include} multidimensional uniqueness and analysis near the degeneracy set. Other possible extensions could be examining different coordinate systems beyond ratio type $x$ and logarithmic type $t$ to see if a Hessian cost manifold can be constructed with a globally defined inverse metric, or if they have a particularly interesting interpretation as statistical manifolds.

 \medskip

\noindent{\bf Acknowledgements.}
The authors thank Elshad Allahyarov,  Sebastian Pardo-Guerra, and  Megan Simons for their valuable comments on an earlier version of the paper.

\medskip

\noindent{\bf Author contributions.} Conceptualization, J.W.;  Methodology, J.W.,  M.Z. and P.B; Software, J.W.; Validation, J.W., M.Z. and P.B.; Formal Analysis, M.Z., P.B. and J.W.; Investigation, J.W., M.Z. and P.B.; Resources, J.W.; Writing-Original Draft Preparation, J.W.; Writing-Review and Editing, M.Z., P.B., and J.W.; Funding Acquisition, J.W.

\end{document}